\documentclass[11pt]{article}
\usepackage{amsmath,amsfonts,amsthm,amssymb,mathtools,enumerate}
\usepackage{mathrsfs, graphicx,color,latexsym, tikz, calc
}
\usepackage[colorlinks=true,
linkcolor=blue,citecolor=blue,
urlcolor=blue]{hyperref}
\usetikzlibrary{shadows}
\usetikzlibrary{patterns,arrows,decorations.pathreplacing}
\newtheorem{theorem}{Theorem}
\newtheorem{lemma}{Lemma}
\newtheorem{problem}{Problem}

\newtheorem{corollary}{Corollary}
\newtheorem{remark}{Remark}

\title{\bf \Large On directed strongly regular Cayley graphs over non-abelian groups with an abelian subgroup of index $2$}

\author{Xueyi Huang$^a$, \ \ Lu Lu$^{b,}$\footnote{Corresponding author.}\setcounter{footnote}{-1}\footnote{\emph{E-mail address:}  huangxymath@163.com (X. Huang), lulugdmath@163.com (L. Lu), jongyook@knu.ac.kr (J. Park).}, \ \  Jongyook Park$^{c}$\\[2mm]
\small $^a$School of Mathematics, East China University of Science and Technology,\\
\small Shanghai, 200237, P. R. China\\
\small $^b$School of Mathematics and Statistics, Central South University,\\ 
\small Changsha, Hunan, 410083, P. R. China\\
\small
$^c$Department of Mathematics, Kyungpook National University,\\ \small Daegu, 41566, Republic of Korea
}

\date{ }
\begin{document}
\maketitle
\begin{abstract}
In 1988, Duval introduced the concept of directed strongly regular graphs, which can be viewed as a directed graph version of strongly regular graphs. Such directed graphs have similar structural and algebraic properties to strongly regular graphs. In the past three decades, it was found that Cayley graphs, especially those over dihedral groups, play a key role in the construction of directed strongly regular graphs.  In this paper, we focus on the characterization of directed strongly regular Cayley graphs over more general  groups. Let $G$ be a non-abelian group with an abelian subgroup of index $2$. We give some necessary conditions for a Cayley graph over $G$ to be directed strongly regular, and characterize the directed strongly regular Cayley graphs over $G$ satisfying specified conditions. This extends some previous results of He and Zhang (2019).

\par\vspace{2mm}

\par\vspace{1mm}

\end{abstract}

\baselineskip=0.202in

\section{Introduction}

Let $\Gamma$ be a directed loopless graph on vertex set $V_\Gamma$. For any two distinct vertices $x,y$ of $\Gamma$, we write $x\rightarrow y$ to  denote that there is an arc from $x$ to $y$, and write $x\leftrightarrow y$ to represent that  $x\rightarrow y$ and $y\rightarrow x$. For convenience, we say that there is an undirected edge between $x$ and $y$ (resp., a directed edge from $x$ to $y$) if $x\leftrightarrow y$ (resp.,  $x\rightarrow y$ but $y\nrightarrow x$).
The \textit{adjacency matrix} of $\Gamma$ is defined as the $(0,1)$-matrix $A_\Gamma=(a_{xy})_{x,y\in V_\Gamma}$ with $a_{xy}=1$ if $x\rightarrow y$, and $a_{xy}=0$ otherwise. The eigenvalues of $A_\Gamma$ are also called the \textit{eigenvalues} of $\Gamma$, and the multiset of eigenvalues of $\Gamma$ is called the \textit{spectrum} of $\Gamma$, and denoted by $\mathrm{Spec}(\Gamma)$.  

In 1988, Duval \cite{DU88} intoduced the concept of directed strongly regular graphs. A \textit{directed strongly regular graph} (\textit{DSRG} for short) with parameters $(n,k,\mu,\lambda,t)$ is a $k$-regular directed loopless graph $\Gamma$ on $n$ vertices such that for every vertex $x\in V_\Gamma$  the number of  $z\in V_\Gamma$ satisfying  $x\leftrightarrow z$ is equal to $t$,  and for every pair of vertices $x,y\in V_\Gamma$ the number of $z\in V_\Gamma$ satisfying $x\rightarrow z\rightarrow y$ is equal to $\lambda$ whenever $x\rightarrow y$, and $\mu$ otherwise.
In other words, a directed loopless graph  $\Gamma$ on $n$ vertices is a DSRG with parameters $(n,k,\mu,\lambda,t)$ if and only if its adjacency matrix $A_\Gamma$ satisfies 
\[
A_\Gamma J=JA_\Gamma=kJ~\mbox{and}~ A_\Gamma^2=tI+\lambda A_\Gamma +\mu(J-I-A_\Gamma),
\]
where $J=J_n$ and $I=I_n$ denote the all-ones matrix and  identity matrix of order $n$, respectively. Clearly, $t\leq k$. A DSRG with parameters $(n,k,\mu,\lambda,k)$ (i.e., $t=k$) is also known as an (undirected) \textit{strongly regular graph} (\textit{SRG} for short) with parameters $(n,k,\lambda,\mu)$. Also, Duval \cite{DU88} proved that a DSRG with parameters  $(n,k,\mu,\lambda,0)$ (i.e., $t=0$) is a doubly regular tournament. For this reason, in this paper, we always assume that  $0<t<k$. It is known that such a DSRG has exactly three distinct eigenvalues and all of them are integers \cite{DU88}.

In the past three decades, the characterization of DSRGs has attracted a lot of interest, and several methods have been developed to investigate the constructions of DSRGs. In \cite{DU88}, Duval presented some basic properties of DSRGs, and constructed several families of DSRGs by using quadratic residue matrices, Kronecker product and block matrices. Klin, Munemasa, Muzychuk, and Zieschang \cite{KMMZ04} obtained new infinite series of DSRGs by using coherent algebras. Hobart and Justin Shaw \cite{HJ99} constructed a new infinite family of DSRGs via Cayley graphs over dihedral groups. Fiedler, Klin, and Muzychuk \cite{FKM02} confirmed the existence of DSRGs for three feasible parameter sets listed by Duval \cite{DU88}.
 Duval and Iourinski \cite{DI03} obtained a new infinite family of DSRGs via Cayley graphs over certain semidirect product groups. Brouwer, Olmez, and Song \cite{BOS12} constructed some families of DSRGs with $t=\mu$ by using antiflags of $1\frac{1}{2}$-designs. Mart\'{i}nez and Araluze \cite{MA10} defined a new combinatorial structure, called  partial sum family, and used it to obtain some families of DSRGs. He and Zhang \cite{HZ19} proved that Cayley graphs whose connection set is a union of some conjugate classes of a group cannot produce DSRGs, and obtained a larger family of DSRGs by generalizing the semidirect product method of Duval and Iourinski  \cite{DI03}.   Very recetly, He, Zhang, and Feng \cite{HZF21} characterized some Cayley DSRGs over the dihedral group $D_{p^\alpha}$, where $p$ is a prime and $\alpha\geq 1$. For more results about the constructions of DSRGs, we refer the reader to \cite{AGOS13,FZ14,FKP99,G16,OS14}.

Let $G$ be a finite group with identity $e$, and let $S$ be a multi-subset of $G\setminus \{e\}$. The \textit{Cayley multigraph}  $\mathrm{Cay}(G,S)$ is  the directed graph with vertex set $G$ in which the number of arcs from $x$ to $y$ is equal to the multiplicity of $yx^{-1}$ in $S$, for all $x,y\in G$. Here $S$ is called the \textit{connection set} of $\mathrm{Cay}(G,S)$. Clearly, $\mathrm{Cay}(G,S)$ is $|S|$-regular. In particular, if $S$ is a subset of $G\setminus \{e\}$, then $\mathrm{Cay}(G,S)$ is called a \textit{Cayley graph}.

In this paper, inspired by the work of He and Zhang \cite{HZ19} for Cayley DSRGs over dihedral groups, we focus on the characterization of Cayley DSRGs over more general groups, namely those non-abelian groups admitting an abelian subgroup of index $2$. For example, the (generalized) dihedral groups and  (generalized) dicyclic groups are such kinds of groups. Let $G$ be a non-abelian group with an abelian subgroup of index $2$. We provide some necessary  conditions  for Cayley graphs over $G$ being DSRGs. Furthermore, we characterize the Cayley DSRGs over $G$ satisfying specified conditions, which generalizes a result of He and Zhang \cite{HZ19} regarding the characterization of a special class of Cayley DSRGs over dihedral groups.

\section{Preliminaries}

Let $G$ be a finite group. Let $\mathbb{C}G$ denote the group algebra of $G$ over $\mathbb{C}$, and  $\mathrm{Irr}(G)$ the set of irreducible characters of $G$. For any multi-subset $X$ of $G$, we denote \[
\overline{X}=\sum_{x\in X} \Delta_X(x)\cdot x\in \mathbb{C}G,
\] 
where
$\Delta_X: X\rightarrow \mathbb{Z}$ is the \textit{multiplicity function} defined by taking $\Delta_X(x)$ as the multiplicity of $x$ in $X$, for all $x\in X$. The following equivalence condition for  Cayley graphs being DSRGs can be deduced directly from the  definitions of Cayley graphs and DSRGs.

\begin{lemma}[\cite{HZ19}]\label{lem::1}
A Cayley graph $\mathrm{Cay}(G,S)$ is a DSRG with parameters $(n,k,\mu,\lambda,t)$ if and only if $|G|=n$, $|S|=k$, and \[\overline{S}^2=te+\lambda \overline{S}+\mu(\overline{G}-e-\overline{S}).\]
In particular, a Cayley graph $\mathrm{Cay}(G,S)$ with $S=S^{-1}$ is a SRG with parameters $(n,k,\lambda,\mu)$ if and only if $|G|=n$, $|S|=k$, and 
\[\overline{S}^2=ke+\lambda \overline{S}+\mu(\overline{G}-e-\overline{S}).\]
\end{lemma}

Let  $\chi\in \mathrm{Irr}(G)$ be an irreducible character of $G$. For any  $\mathcal{X}=\sum_{g\in G}c_g g\in \mathbb{C}G$, we denote
\[
\chi(\mathcal{X})=\sum_{g\in G}c_g \chi(g)\in \mathbb{C}.
\] 
If $G$ is abelian, then $\chi(\mathcal{X}\cdot \mathcal{Y})=\chi(\mathcal{X})\cdot \chi(\mathcal{Y})$ for 
all $\mathcal{X},\mathcal{Y}\in \mathbb{C}G$. Moreover, we have the following result.

\begin{lemma}[\cite{HZ19}]\label{lem::2}
Let $G$ be an abelian group and $\mathcal{X},\mathcal{Y}\in \mathbb{C}G$. Then $\mathcal{X}=\mathcal{Y}$ if and only if $\chi(\mathcal{X})=\chi(\mathcal{Y})$ for all $\chi\in\mathrm{Irr}(G)$.
\end{lemma}

The following lemma due to Babai \cite{B79} provides an expression for the  eigenvalues of  Cayley multigraphs over abelian groups in terms of irreducible characters.

\begin{lemma}[\cite{B79}]\label{lem::3}
Let $G$ be an abelian group, and let $S$ be a multi-subset of $G$. Then the eigenvalues of  $\mathrm{Cay}(G,S)$  are $\chi(\overline{S})$, for  $\chi\in \mathrm{Irr}(G)$.
\end{lemma}

\begin{lemma}[\cite{KMMZ04}]\label{lem::4}
Let $\Gamma$ be a regular non-empty directed multigraph without undirected edges. Then $\Gamma$ has at least one non-real eigenvalue.
\end{lemma}

By Lemmas \ref{lem::3} and \ref{lem::4}, we deduce the following result immediately.

\begin{corollary}[\cite{HZ19}]\label{cor::1}
Let  $G$ be an abelian group. If $\Gamma=\mathrm{Cay}(G,S)$ is a Cayley multigraph with $S\neq S^{-1}$, then  $\Gamma$ has at least one non-real eigenvalue.
\end{corollary}

\begin{lemma}\label{lem::5}
If $\Gamma$ is a connected undirected multigraph with second largest eigenvalue at most $0$, then the underlying graph of $\Gamma$ is complete multipartite.
\end{lemma}
\begin{proof}
Let $\Gamma_1$ denote the underlying graph of $\Gamma$. To prove that $\Gamma_1$ is a complete multipartite graph, it suffices to prove that $\Gamma_1$ does not contain any one of $I_1$, $I_2$, $I_3$, $I_4$ (as shown in Figure \ref{fig::1}) as an induced subgraph. Indeed, if $I_1$, $I_2$, $I_3$ or $I_4$ is an induced subgraph of  $\Gamma_1$, then $\Gamma$ has at least one of the following matrices as a principal submatrix: 
$$
M_1=\begin{pmatrix}
0 &0 &0\\
0&0& a\\
0& a &0
\end{pmatrix},
M_2=\begin{pmatrix}
0 &a &0 &0\\
a&0& 0 &0\\
0& 0 &0 &b\\
0 &0& b&0
\end{pmatrix},
M_3=\begin{pmatrix}
0 &a &0 &0\\
a&0& b &0\\
0& b &0 &c\\
0 &0& c&0
\end{pmatrix},
M_4=\begin{pmatrix}
0 &a &c &0\\
a&0& b &0\\
c& b &0 &d\\
0 &0& d&0
\end{pmatrix},
$$
where $a,b,c,d$ are positive integers. However, it is easy to check that each $M_i$ ($i=1,2,3,4$) has at least two non-negative eigenvalues. Thus, by Cauchy Interlacing Theorem, $\Gamma$ has at least two non-negative eigenvalues, a contradiction.
\end{proof}

\begin{figure}[t]
    \centering
\includegraphics[width=10cm]{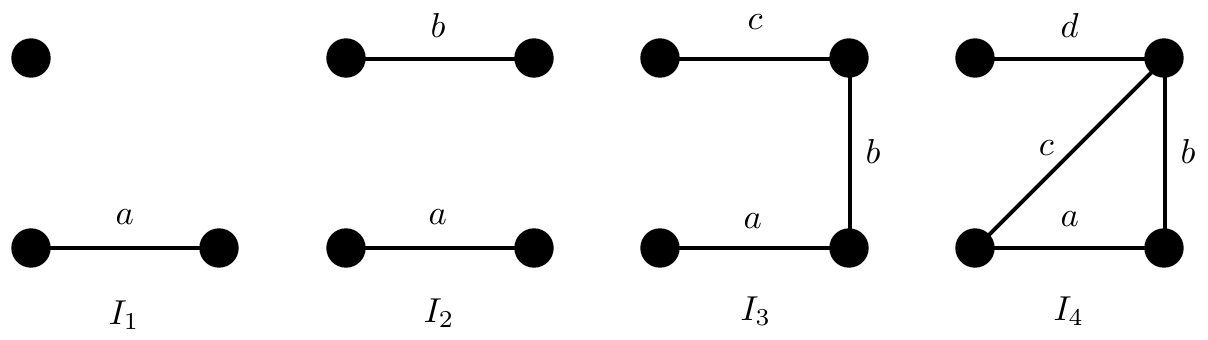}
    \caption{The graphs $I_1$, $I_2$, $I_3$ and $I_4$.}
    \label{fig::1}
\end{figure}

\section{Main results}
In this section, we focus on the characterization of Cayley DSRGs over non-abelian groups admitting an abelian subgroup of index $2$.

Let $G$ be a non-abelian group of order $2n$ ($n>1$) with an abelian subgroup $A$ of index $2$. Then it is not difficult to see that $G$ is of the form 
\begin{equation}\label{equ::1}
    G=\langle A,\beta\mid \beta^2=\alpha, \beta a\beta^{-1}=f(a)~\mbox{for all $a\in A$}\rangle,
\end{equation}
where $\alpha$ is any given element of $A$, and $f$ is any given automorphism of $A$ with order $2$ such that $f(\alpha)=\alpha$. Clearly, $G=A\cup A\beta$. In particular, if $A=\langle a\mid a^n=1 \rangle$ is a cyclic group and $f(a)=a^{-1}$, then $G$ is known as a \textit{dihedral group} when $\alpha=e$, and a \textit{dicyclic group} when  $\alpha=a^{n/2}$ ($n$ is even). 

In what follows, we always assume that $\Gamma$ is a Cayley graph over the group $G$ given by \eqref{equ::1}, that is,  $\Gamma=\mathrm{Cay}(G,X\cup Y\beta)$ with $X\subseteq A\setminus\{e\}$ and $Y\subseteq A$. For any $\mathcal{X}=\sum_{a\in A}c_a\cdot a\in \mathbb{C}A$, we denote
$f(\mathcal{X})=\sum_{a\in A}c_a\cdot f(a)\in \mathbb{C}A.
$
Clearly,  $f(\mathcal{X}\cdot \mathcal{Y})=f(\mathcal{X})\cdot f(\mathcal{Y})=f(\mathcal{Y})\cdot f(\mathcal{X})=f(\mathcal{Y}\cdot \mathcal{X})$ for 
any $\mathcal{X},\mathcal{Y}\in \mathbb{C}A$. Then, by Lemma \ref{lem::1}, we have the following criterion for $\Gamma$ being a DSRG.

\begin{lemma}\label{lem::6}
The Cayley graph $\Gamma$ is a DSRG with parameters $(2n,k,\mu,\lambda,t)$ if and only if $|X|+|Y|=k$, and 
\begin{align}
&\overline{X}^2+\overline{Y}f(\overline{Y})\alpha=(t-\mu)e+(\lambda-\mu)\overline{X}+\mu\overline{A},\label{equ::2}\\
&\overline{X}\cdot \overline{Y}+\overline{Y}f(\overline{X})=(\lambda-\mu)\overline{Y}+\mu\overline{A}.\label{equ::3}
\end{align}
\end{lemma}

For any subsets $X_1,X_2$ of $A$, we denote by $X_1\sqcup X_2$ the multiset consisting of all elements in $X_1$ and $X_2$. Clearly, if $X_1\cap X_2=\emptyset$, then $X_1\sqcup X_2=X_1\cup X_2$. Using Lemma \ref{lem::6}, we first give a necessary condition for $\Gamma$ being a DSRG in terms of the irreducible characters of  $A$.

\begin{theorem}\label{thm::1}
Suppose that the Cayley graph $\Gamma$ is a DSRG with parameters $(2n,k,\mu,\lambda,t)$. Then the following statements hold.
\begin{enumerate}[(i)]
\item For any non-trivial  $\chi\in\mathrm{Irr}(A)$, if $\chi(\overline{Y})\neq 0$ then  $\chi(\overline{X}+f(\overline{X}))=\lambda-\mu$, and if $\chi(\overline{Y})=0$ then $\chi(\overline{X}),\chi(f(\overline{X}))\in \{|X|-|Y|,(\lambda-\mu)-(|X|-|Y|)\}$. 
\item The Cayley multigraph $\Gamma':=\mathrm{Cay}(A,X\sqcup f(X))$ is undirected, i.e., $X\sqcup f(X)$ is inverse closed, and all  possible eigenvalues of $\Gamma'$ are $2|X|$ (with multiplicity $1$), $\lambda-\mu$, $2(|X|-|Y|)$ and $2[(\lambda-\mu)-(|X|-|Y|)]$. 
In particular, if $\chi(\overline{Y})\neq 0$ for all non-trivial $\chi\in \mathrm{Irr}(A)$, then $\Gamma'$ has exactly two distinct eigenvalues $2|X|$ (with multiplicity $1$) and $\lambda-\mu$. 
\end{enumerate}
\end{theorem}

\begin{proof}
Let $\chi\in \mathrm{Irr}(A)$ be any non-trivial irreducible character of $A$. Then $\chi(\overline{A})=0$, and it follows from \eqref{equ::3} that
\[\chi(\overline{Y})\left(\chi(\overline{X}+f(\overline{X}))-(\lambda-\mu)\right)=0.\]
If $\chi(\overline{Y})\ne0$, then $\chi(\overline{X}+f(\overline{X}))=\lambda-\mu$, as required. If  $\chi(\overline{Y})=0$, then from \eqref{equ::2} we obtain 
\begin{equation}\label{equ::4}
    \chi(\overline{X})^2=(t-\mu)+(\lambda-\mu)\chi(\overline{X}).
\end{equation}
On the other hand, by applying the trivial character to both sides of \eqref{equ::2} and \eqref{equ::3}, we obtain 
$|X|^2+|Y|^2=(t-\mu)+(\lambda-\mu)|X|+\mu n$ and
$2|X|\cdot |Y|=(\lambda-\mu)|Y|+\mu n$, respectively. It follows that \begin{equation}\label{equ::5}
(|X|-|Y|)^2=(t-\mu)+(\lambda-\mu)(|X|-|Y|).
\end{equation}
Combining \eqref{equ::4} with \eqref{equ::5}, we get $\chi(\overline{X})=|X|-|Y|$ or $(\lambda-\mu)-(|X|-|Y|)$, as desired. Furthermore, by applying $f$ to both sides of \eqref{equ::2}, we have 
$f(\overline{X})^2+\overline{Y}f(\overline{Y})\alpha=(t-\mu)e+(\lambda-\mu)f(\overline{X})+\mu\overline{A}$,
and hence $\chi(f(\overline{X}))^2=(t-\mu)+(\lambda-\mu)\chi(f(\overline{X}))$.
Thus we also can deduce that $\chi(f(\overline{X}))\in\{|X|-|Y|,(\lambda-\mu)-(|X|-|Y|)\}$. This proves (i).

Now suppose  $\Gamma':=\mathrm{Cay}(A,X\sqcup f(X))$. According to (i), for any non-trivial $\chi\in \mathrm{Irr}(A)$, we have $\chi(\overline{X}+f(\overline{X}))\in  \left\{\lambda-\mu, 2(|X|-|Y|), 2[(\lambda-\mu)-(|X|-|Y|)]\right\}$. 
By Lemma \ref{lem::3} and by the arbitrariness of $\chi$, we conclude that all possible eigenvalues of  $\Gamma'$ other than  $2|X|$ are $\lambda-\mu$, $2(|X|-|Y|)$ and $2[(\lambda-\mu)-(|X|-|Y|)]$.   Since $\Gamma'$ has only integer eigenvalues, by Corollary \ref{cor::1}, $X\sqcup f(X)$ is inverse closed, and hence $\Gamma'$ is undirected. Moreover, if $\chi(\overline{Y})\neq 0$ for all non-trivial $\chi\in \mathrm{Irr}(A)$, then the above arguments imply that  $\Gamma'$ has exactly two distinct eigenvalues $2|X|$  (with multiplicity $1$)  and $\lambda-\mu$.  Thus (ii) follows.
\end{proof}

\begin{remark}\label{rem::1}
\emph{It is worth mentioning that all Cayley multigraphs over abelian groups with only integer eigenvalues have been charaterized in \cite{DKMA13}. Let $A$ be an abelian group. For any $x\in A$, let $\mathcal{O}_x=\{y\in A\mid \langle y \rangle=\langle x \rangle\}$.
Then a Cayley multigraph $\mathrm{Cay}(A,S)$  has only integer eigenvalues if and only if the multi-subset $S$ can be expressed as a non-negative integer combination of $\mathcal{O}_x$ for $x\in A$. Therefore, by Theorem \ref{thm::1}, if the Cayley graph $\Gamma$ is a DSRG, then 
$\overline{X}+f(\overline{X})=\sum_{x\in A}m_x\overline{\mathcal{O}_x}$,
where $m_x\in \{0,1,2\}$ for all $x\in A$.
}
\end{remark}

As an application of Theorem \ref{thm::1}, we obtain the following result.

\begin{corollary}\label{cor::2}
If the Cayley graph $\Gamma$ is a DSRG, then one of the following statements holds:
\begin{enumerate} [(i)]
    \item there exists some non-trivial $\chi\in \mathrm{Irr}(G)$ such that $\chi(\overline{Y})=0$;
    \item $X\cap f(X)=\emptyset$, $X\cup f(X)=A\setminus\{e\}$, and 
    $\frac{n-\sqrt{2n-1}}{2}<|Y|<\frac{n+\sqrt{2n-1}}{2}$.
\end{enumerate}
In particular, if $Y\in\left\{\{e\},A\setminus\{e\}\right\}$ and $n\ge 5$, then  $\Gamma$ cannot be a DSRG. 
\end{corollary}
\begin{proof}Suppose that the Cayley graph $\Gamma$ is a DSRG with parameters $(2n,k,\mu,\lambda,t)$, where $0<t<k$. Assume that  $\chi(\overline{Y})\neq 0$ for all non-trivial $\chi\in \mathrm{Irr}(A)$. By Theorem \ref{thm::1}, the Cayley multigraph $\Gamma':=\mathrm{Cay}(A,X\sqcup f(X))$ is undirected, and has exactly two distinct eigenvalues $2|X|$ (with multiplicity $1$) and $\lambda-\mu$. This implies that the underlying graph of $\Gamma'$, which is the simple graph obtained from $\Gamma'$ by deleting all possible parallel edges, must be complete. Then, by Lemma \ref{lem::3}, one can easily deduce that $\overline{X}+f(\overline{X})=\overline{A}-e$ or $2(\overline{A}-e)$. 

If $\overline{X}+f(\overline{X})=\overline{A}-e$, then we must have $X\cap f(X)=\emptyset$, $X\sqcup f(X)=X\cup f(X)=A\setminus\{e\}$, $|X|=(n-1)/2$ and $\lambda-\mu=-1$. Combining this with \eqref{equ::3} yields that $\mu=|Y|$ and $\lambda=|Y|-1$. Then, by applying the trivial character to both sides of \eqref{equ::2}, we get
$t=|X|(|X|+1)-|Y|(n-1-|Y|)$. Also note that $t<k=|X|+|Y|$. Thus we have $|X|^2<|Y|(n-|Y|)$, and it follows that
$\frac{n-\sqrt{2n-1}}{2}<|Y|<\frac{n+\sqrt{2n-1}}{2}$ because $|X|=(n-1)/2$.

If $\overline{X}+f(\overline{X})=2(\overline{A}-e)$, then  $X=f(X)=A\setminus\{e\}$. In this situation, $\Gamma$ is undirected, and hence $t=k$, contrary to our assumption that $t<k$. 

Therefore, we conclude that one of the statements in (i) and (ii) holds.  

Now suppose that $Y\in\left\{\{e\},A\setminus\{e\}\right\}$. Clearly, $\chi(\overline{Y})\in \{1,-1\}$ for all non-trivial $\chi\in \mathrm{Irr}(A)$. If $\Gamma$ is a DSRG, then from the above arguments we obtain $\frac{n-\sqrt{2n-1}}{2}<|Y|<\frac{n+\sqrt{2n-1}}{2}$. However, this is impossible because  $|Y|\in\{1,n-1\}$ and $n\geq 5$. 
\end{proof}

According to Theorem \ref{thm::1}, if $|X|=|Y|$, then   $\chi(\overline{X}+f(\overline{X}))$ has only three possible values   $0$, $\lambda-\mu$ and $2(\lambda-\mu)$, for all non-trivial $\chi\in\mathrm{Irr}(A)$.  Based on this observation,  we give a sufficient and necessary condition for the Cayley graph $\Gamma$ being a DSRG under certain assumptions.

\begin{theorem}\label{thm::2}
Assume that $X\cap f(X)=Y\cap f(Y)=\emptyset$ and $|X|=|Y|$. Let  $B=A\setminus(X\cup f(X))$. Then the Cayley graph $\Gamma$ is a DSRG if and only if the following conditions hold:
\begin{enumerate}[(i)]
    \item $B$ is a subgroup of $A$,
    \item both $X$ and $Y$ are a union of some cosets of $B$ in $A$,
    \item $X\cup f(X)=Y\cup f(Y)$ and $\overline{X}f(\overline{X})=\overline{Y}f(\overline{Y})$.
\end{enumerate}
In this case, $\Gamma$ has the parameters $(2n,n-\ell,\frac{n-\ell}{2},\frac{n-3\ell}{2},\frac{n-\ell}{2})$, where $\ell=|B|$.
\end{theorem}
\begin{proof}
Clearly, $f(B)=B$. First assume that $\Gamma$ is a DSRG with parameters $(2n,k,\mu,\lambda,t)$. As $|X|=|Y|$, by applying the trivial character to both sides of  \eqref{equ::2} and \eqref{equ::3}, we obtain
$2|X|^2=(t-\mu)+(\lambda-\mu)|X|+\mu n$ and $2|X|^2=(\lambda-\mu)|X|+\mu n$, respectively.  Then $t=\mu$, and \eqref{equ::2} becomes 
\begin{equation}\label{equ::6}
\overline{X}^2+\overline{Y}f(\overline{Y})\alpha=(\lambda-\mu)\overline{X}+\mu\overline{A}.
\end{equation}
Recall that $f$ is an automorphism of $A$ with order $2$ such that $f(\alpha)=\alpha$. By applying $f$ to both sides of \eqref{equ::6}, we obtain 
\begin{equation}\label{equ::7}
f(\overline{X})^2+\overline{Y}f(\overline{Y})\alpha=(\lambda-\mu)f(\overline{X})+\mu\overline{A}.
\end{equation}
Let $\Gamma':=\mathrm{Cay}(A,X\cup f(X))$. As $X\cap f(X)=\emptyset$, by Theorem \ref{thm::1}, $\Gamma'$ is an undirected Cayley graph, i.e., $X\cup f(X)$ is inverse closed, and all possible eigenvalues of $\Gamma'$ are  $2|X|$ (with multiplicity $1$), $0$, $\lambda-\mu$ and $2(\lambda-\mu)$.  Moreover, since the sum of all eigenvalues of $\Gamma'$ is equal to $0$, we have $\lambda<\mu$, and hence the second largest eigenvalue of $\Gamma'$ is at most $0$. This implies that  $\Gamma'$ is exactly a complete multipartite graph. We shall prove that $B=A\setminus (X\cup f(X))$ is a subgroup of $A$. Let $b_1,b_2\in B$. If $e\in \{b_1,b_2\}$, then $b_2b_1^{-1}\in B$ because $B^{-1}=B$ due to $X\cup f(X)$ is inverse closed. If $e\notin\{b_1,b_2\}$, then $b_1$ and $b_2$ are not adjacent, since otherwise $K_1\cup K_2$ would be an induced subgraph of $\Gamma'$, which is impossible. Then $b_2b_1^{-1}\notin X\cup f(X)$, and so $b_2b_1^{-1}\in B$. By the arbitrariness of $b_1,b_2\in B$, we conclude that $B$ is a subgroup of $A$. Since $B=A\setminus(X\cup f(X))$ and $X\cap f(X)=\emptyset$, from \eqref{equ::6} and \eqref{equ::7} we obtain 
\begin{equation}\label{equ::8}
    (\overline{A}-\overline{B})(f(\overline{X})-\overline{X})=(\overline{X}+f(\overline{X}))(f(\overline{X})-\overline{X})=(\lambda-\mu)(f(\overline{X})-\overline{X}).
\end{equation} 
As $X\neq f(X)$, by Lemma \ref{lem::2}, there exists some nontrivial character $\chi\in\mathrm{Irr}(A)$ such that $\chi(\overline{X})\neq \chi(f(\overline{X}))$. By applying $\chi$ to both sides of \eqref{equ::8},  we obtain $\chi(\overline{B})=\mu-\lambda$. On the other hand, since $B$ is a subgroup of $A$, we have  $\overline{B}^2=|B|\overline{B}$ and $\chi(\overline{B})^2=|B|\chi(\overline{B})$. Thus $|B|=\mu-\lambda$ due to $\chi(\overline{B})=\mu-\lambda\neq 0$. Recall that $f(\overline{X})=\overline{A}-\overline{B}-\overline{X}$. Again by \eqref{equ::8}, 
\begin{equation}\label{equ::9}
    (\overline{A}-\overline{B})(\overline{A}-\overline{B}-2\overline{X})=(\lambda-\mu)(\overline{A}-\overline{B}-2\overline{X}).
\end{equation} 
Combining this with  $\overline{A}^2=n\overline{A}$, $\overline{B}^2=|B|\cdot\overline{B}$, $\overline{A}\cdot \overline{B}=|B|\cdot\overline{A}$, $\overline{A}\cdot \overline{X}=|X|\cdot \overline{A}$, $|B|=\mu-\lambda$ and $2|X|+|B|=n$ yields that
$$
\overline{B}\cdot\overline{X}=|B|\cdot \overline{X}.
$$
Therefore, we may conclude that $X$  is a union of some cosets of $B$ in $A$. Furthermore, from $2|X|^2=(\lambda-\mu)|X|+\mu n$, $|B|=\mu-\lambda$ and $2|X|+|B|=n$ we can deduce that $|X|=\mu$. Then, by \eqref{equ::3} and the fact that $\overline{X}+f(\overline{X})=\overline{A}-\overline{B}$, we have
$$
(\overline{A}-\overline{B})\overline{Y}=(\lambda-\mu)\overline{Y}+\mu\overline{A}.
$$
Combining this with $\overline{A}\cdot \overline{Y}=|Y|\overline{A}=|X|\overline{A}=\mu\overline{A}$, we obtain
$$
\overline{B}\cdot\overline{Y}=(\mu-\lambda)\overline{Y}=|B|\cdot \overline{Y}.
$$
This implies that $Y$ is a union of some cosets of $B$ in $A$. As $f(B)=B$ and $Y\cap f(Y)=\emptyset$,  we have $B\cap (Y\cup f(Y))=\emptyset$. Then it follows from $|Y|+|f(Y)|=2|Y|=2|X|=n-|B|$ that $Y\cup f(Y)=A\setminus B=X\cup f(X)$.  Recall that $f(\alpha)=\alpha$. We have $\alpha\in B$, and so $\overline{Y}\alpha=\overline{Y}$. Then
\eqref{equ::6} becomes 
$$
\overline{X}(\overline{A}-\overline{B}-f(\overline{X}))+\overline{Y}f(\overline{Y})=(\lambda-\mu)\overline{X}+\mu \overline{A},
$$
or equivalently, 
$$
\overline{X}f(\overline{X})=\overline{Y}f(\overline{Y}).
$$
Therefore, we conclude that (i)--(iii) hold. Moreover, from $|X|=\mu=\frac{n-|B|}{2}$ and $|B|=\mu-\lambda$ we deduce that $\lambda=\frac{n-3|B|}{2}$. Hence, $\Gamma$ has the parameters $(2n,n-\ell,\frac{n-\ell}{2},\frac{n-3\ell}{2},\frac{n-\ell}{2})$, where $\ell=|B|$. 

Conversely, if (i)--(iii) hold, it is routine to verify that $X$ and $Y$ satisfy the equations in \eqref{equ::2} and \eqref{equ::3} with $t=\mu=\frac{n-\ell}{2}$ and $\lambda=\frac{n-3\ell}{2}$, where $\ell=|B|$. Therefore, by Lemma \ref{lem::6}, $\Gamma$ is a DSRG.
\end{proof}

Assume that $Y\in\{X,f(X)\}$ and $X\cap f(X)=\emptyset$. Since $f(B)=B$, we see that $X$ is a union of some cosets of $B$ in $A$ if and only if $f(X)$ is a union of some cosets of $B$ in $A$. Then from Theorem \ref{thm::2} we obtain the following corollary immediately.

\begin{corollary}\label{cor::3}
Assume that $Y\in\{X,f(X)\}$ and $X\cap f(X)=\emptyset$. Let $B=A\setminus(X\cup f(X))$. Then the Cayley graph $\Gamma$ is a DSRG if and only if $B$ is a subgroup of $A$, and $X$ is a union of some cosets of $B$ in $A$. In this case, $\Gamma$ has the parameters $(2n,n-\ell,\frac{n-\ell}{2},\frac{n-3\ell}{2},\frac{n-\ell}{2})$, where $\ell=|B|$.
\end{corollary}

\begin{remark}\label{rem::2}
\emph{
If $G$ is a dihedral group, i.e., $G=\langle A,\beta\rangle$ with 
$A=\langle a\mid a^n=1\rangle$, $\beta^2=\alpha=e$ and $\beta^{-1}a\beta=f(a)=a^{-1}$, then  Corollary \ref{cor::3} determines all  Cayley DSRGs of the form  $\mathrm{Cay}(G,X\cup X\beta)$, where $X\subseteq A\setminus\{e\}$ and $X\cap X^{-1}=\emptyset$. This concides with the result of He and Zhang in \cite[Theorem 6.10]{HZ19}.}
\end{remark}

Under the assumption of Theorem \ref{thm::2}, if $\Gamma$ is a DSRG, then its parameters are of the form $(2n,n-\ell,\frac{n-\ell}{2},\frac{n-3\ell}{2},\frac{n-\ell}{2})$. Naturally, we propose the following problem.

\begin{problem}\label{prob::1}
For any fixed positive integers $n$ and $\ell$ with $n-\ell$ being even, determine all Cayley DSRGs over $G$ with parameters $(2n,n-\ell,\frac{n-\ell}{2},\frac{n-3\ell}{2},\frac{n-\ell}{2})$.
\end{problem}

For $\ell=1$, we can give an answer to  Problem \ref{prob::1} immediately.  
\begin{theorem}\label{thm::3}
Let $n\geq 3$ be an odd number. The Cayley graph $\Gamma$ is a DSRG with parameters $(2n,n-1,\frac{n-1}{2},\frac{n-3}{2},\frac{n-1}{2})$ if and only if the following conditions hold:
\begin{enumerate}[(i)]
    \item $|X|=|Y|=\frac{n-1}{2}$,
    \item $X\cap f(X)=\emptyset$ and $X\cup f(X)=A\setminus \{e\}$,
    \item $\alpha=e$,
    \item $\overline{X}f(\overline{X})=\overline{Y}f(\overline{Y})$.
\end{enumerate}
\end{theorem}
\begin{proof}
First assume that $\Gamma$ is a DSRG with parameters $(2n,n-1,\frac{n-1}{2},\frac{n-3}{2},\frac{n-1}{2})$. According to \eqref{equ::2} and \eqref{equ::3}, we have
\begin{equation}\label{equ::10}
    \overline{X}^2+\overline{Y}f(\overline{Y})\alpha=-\overline{X}+\frac{n-1}{2}\overline{A},
\end{equation}
\begin{equation}\label{equ::11}
    \overline{X}\cdot \overline{Y}+\overline{Y}f(\overline{X})=-\overline{Y}+\frac{n-1}{2}\overline{A}.
\end{equation}
By applying the trivial character to both sides of \eqref{equ::10} and \eqref{equ::11}, we obtain
\begin{equation}\label{equ::12}
    |X|^2+|Y|^2=-|X|+\frac{n(n-1)}{2}~~\mbox{and}~~2|X| |Y|=-|Y|+\frac{n(n-1)}{2}.
\end{equation}
Hence,  $|X|=|Y|-1$ or $|X|=|Y|$. If $|X|=|Y|-1$, then from \eqref{equ::12} we obtain $|Y|=\frac{n}{2}$, contrary to the parity of $n$. Therefore, we must have $|X|=|Y|=\frac{n-1}{2}$ by \eqref{equ::12}. This proves (i).

Now consider (ii). Let $\Gamma'={\rm Cay}(A,X\sqcup f(X))$, $\Gamma_1'={\rm Cay}(A,X\cup f(X))$ and $\Gamma_2'={\rm Cay}(A,X\cap f(X))$. Note that $\Gamma_1'$ is the underlying graph of $\Gamma'$. Let $B=A\setminus(X\cup f(X))$ and $H=X\cap f(X)$. By Theorem \ref{thm::1},  the second largest eigenvalue of $\Gamma'$ is at most $0$. Then Lemma \ref{lem::5} implies that $\Gamma_1'$ is a complete multipartite graph, and hence $B$ is a subgroup of $A$ by a similar argument as in the proof of Theorem \ref{thm::2}. Suppose $|B|=l+1$. Then the spectrum of $\Gamma_1'$ is
\begin{equation}\label{equ::13}
    \mathrm{Spec}(\Gamma_1')=\left\{n-1-l,[0]^{\frac{ln}{l+1}},[-l-1]^{\frac{n}{l+1}-1}\right\}.
\end{equation}
On the other hand, let $X_c=A\setminus(X\cup e)$. According to \eqref{equ::10} and \eqref{equ::11}, we have
$$
\overline{X_c}^2+\overline{Y}f(\overline{Y})\alpha=-\overline{X_c}+\frac{n-1}{2}\overline{A}  ~~\mbox{and}~~
      \overline{X_c}\cdot \overline{Y}+\overline{Y}f(\overline{X_c})=-\overline{Y}+\frac{n-1}{2}\overline{A},
$$
which implies that ${\rm Cay}(G,X_c\cup Y\beta)$ is a DSRG with parameters $(2n,n-1,\frac{n-1}{2},\frac{n-3}{2},\frac{n-1}{2})$. By the above arguments,  the undirected Cayley graph ${\rm Cay}(A,X_c\cup f(X_c))$ is also a complete multipartite graph, and so $A\setminus(X_c\cup f(X_c))$ is a subgroup of $A$. Note that $X_c\cup f(X_c)=(A\setminus (X\cup \{e\}))\cup (A\setminus(f(X)\cup\{e\}))=A-(H\cup \{e\})$. Thus $H\cup \{e\}$ is a subgroup of $A$. Since $n-1=|X|+|f(X)|=|X\cup f(X)|+|X\cap f(X)|=|A|-|B|+|H|=n-(l+1)+|H|$,  we have $|H|=l$. If $H\neq \emptyset$, i.e., $l\geq 1$, we see that $\Gamma_2'=\mathrm{Cay}(A,H)$ is the union of $n/(l+1)$ copies of $\mathrm{Cay}(H\cup \{e\},H)\cong K_{l+1}$, and hence the spectrum of $\Gamma_2'$ is
\begin{equation}\label{equ::14}
\mathrm{Spec}(\Gamma_2')=\left\{[l]^{\frac{n}{l+1}},[-1]^{\frac{ln}{l+1}}\right\}.
\end{equation}
By Lemma \ref{lem::3}, each non-trivial eigenvalue of $\Gamma'$ is the summation of an eigenvalue of $\Gamma_1'$ (as shown in \eqref{equ::13}) and an eigenvalue of $\Gamma_2'$ (as shown in \eqref{equ::14}). Since $\Gamma'$ is an $(n-1)$-regular multigraph with the second largest eigenvalue at most $0$, we assert that the spectrum of $\Gamma'$ is of the form
$$\mathrm{Spec}(\Gamma')=\left\{n-1,[0]^{\frac{n}{l+1}-1},[-1]^{\frac{ln}{l+1}}\right\}.$$
By considering the trace of $A_{\Gamma'}$, we obtain $(n-1)-ln/(l+1)=0$, and hence  $l=n-1$. However, this implies that $|B|=l+1=n$ and $X=\emptyset$, which contradicts (i) due to $n\geq 3$. Therefore,  $H=X\cap f(X)=\emptyset$. Combining this with (i), we obtain  $X\cup f(X)=A\setminus\{e\}$. 

As $\alpha\not\in X\cup f(X)$,  from (ii) we immediately deduce that $\alpha=e$. Thus (iii) holds. According to (ii) and (iii), we see that \eqref{equ::10} becomes  
$$\overline{X}(\overline{A}-e-f(\overline{X}))+\overline{Y}f(\overline{Y})=-\overline{X}+\frac{n-1}{2}\overline{A}.$$
Combining this with $\overline{X}\cdot \overline{A}=\frac{n-1}{2}\overline{A}$ yields that $\overline{X}f(\overline{X})=\overline{Y}f(\overline{Y})$. This proves  (iv).

Conversely, if (i)-(iv) hold, it is routine to verify that $X$ and $Y$ satisfy \eqref{equ::10} and \eqref{equ::11}. Then, by  Lemma \ref{lem::6}, $\Gamma$ is a DSRG with parameters $(2n,n-1,\frac{n-1}{2},\frac{n-3}{2},\frac{n-1}{2})$.
\end{proof}

By Theorem \ref{thm::3}, we obtain the following corollary immediately.
\begin{corollary}\label{cor::4}
If $\alpha\ne e$, then there is no Cayley DSRG over $G$ with parameters $(2n,n-1,\frac{n-1}{2},\frac{n-3}{2},\frac{n-1}{2})$.
\end{corollary}

Inspired by the conclusion of Corollary \ref{cor::3}, we now consider another interesting situation   $Y\in\{A\setminus X,A\setminus f(X)\}$. 

\begin{theorem}\label{thm::4}
Assume that $Y\in \{A\setminus X,A\setminus f(X)\}$ and $X\cap f(X)=\emptyset$. Let $B=A\setminus(X\cup f(X))$. Then the Cayley graph $\Gamma$ is a DSRG if and only if $B$ is a subgroup of $A$, and $X$ is a union of some cosets of $B$ in $A$. In this case, $\Gamma$ has the parameters $(2n,n,\frac{n+\ell}{2},\frac{n-\ell}{2},\frac{n+\ell}{2})$, where $\ell=|B|$. 
\end{theorem}

\begin{proof}
Note that $f(B)=B$, and  $X$ is a union of some cosets of $B$ in $A$ if and only if $f(X)$ is a union of some cosets of $B$ in $A$. By symmetry, we only need to prove the theorem for $Y=A\setminus X$. 
First assume that $\Gamma$ is a DSRG with parameters $(2n,k,\mu,\lambda,t)$. Putting $Y=A\setminus  X$ in \eqref{equ::2} and \eqref{equ::3}, we obtain
\begin{align}
&\overline{X}^2+\overline{X}f(\overline{X})\alpha=(t-\mu)e+(\lambda-\mu)\overline{X}+(2|X|+\mu-n)\overline{A},\label{equ::15}\\ 
&\overline{X}^2+\overline{X}f(\overline{X})=(\lambda-\mu)\overline{X}+(2|X|-\lambda)\overline{A}.\label{equ::16}
\end{align}
Recall that $f$ is an automorphism of $A$ with order $2$. By applying $f$ to both sides of \eqref{equ::16}, we have
\begin{equation}\label{equ::17}
f(\overline{X})^2+\overline{X}f(\overline{X})=(\lambda-\mu)f(\overline{X})+(2|X|-\lambda)\overline{A}.
\end{equation}
Combining \eqref{equ::16} and  \eqref{equ::17} yields that
\begin{equation*}
    (\overline{X}+f(\overline{X}))^2=(\lambda-\mu)(\overline{X}+f(\overline{X}))+2(2|X|-\lambda)\overline{A}.
\end{equation*}
Let $U=X\cup f(X)$. Then we have
\begin{equation}\label{equ::18}
\overline{U}^2=(4|X|-2\lambda)  e+(4|X|-\lambda-\mu)\overline{U}+(4|X|-2\lambda)(\overline{A}-e-\overline{U}).
\end{equation} 
On the other hand, since $X\cap f(X)=\emptyset$, by Theorem \ref{thm::1},  $\Gamma':=\mathrm{Cay}(A,U)$ is an undirected Cayley graph, and hence $U=U^{-1}$. By comparing the coefficient of $e$ in both sides of \eqref{equ::18}, we obtain $2|X|=4|X|-2\lambda$, i.e., $|X|=\lambda$. Combining Lemma \ref{lem::1} and \eqref{equ::18}, we conclude that either $U=A\setminus \{e\}$, or $\Gamma'$ is a SRG with parameters $(n,2\lambda,3\lambda-\mu,2\lambda)$.

If $U=A\setminus\{e\}$, then $B=\{e\}$ is a trivial subgroup of $A$, and $X$ is naturally a union of some cosets of $B$.  Note that $n-1=|U|=2|X|=2\lambda$, i.e., $\lambda=\frac{n-1}{2}$. Putting $\overline{U}=\overline{A}-e$ in  \eqref{equ::18}, we obtain $(\mu-\frac{n+1}{2})(\overline{A}-e)=0$, and hence $\mu=\frac{n+1}{2}$. Also, since $f(\alpha)=\alpha$ and $X\cap f(X)=\emptyset$,  we have $\alpha=e$. Then it follows from \eqref{equ::15} and \eqref{equ::16} that $t=\mu=\frac{n+1}{2}$. Therefore, in this situation, the parameters of  $\Gamma$ are  $(2n,n-1,\frac{n+1}{2},\frac{n-1}{2},\frac{n+1}{2})$.

If $\Gamma'=\mathrm{Cay}(A,U)$ is a SRG with parameters  $(n,2\lambda,3\lambda-\mu,2\lambda)$, then it would be a complete multipartite graph. As in the proof of Theorem \ref{thm::2}, we can prove that $B=A\setminus U$ is a subgroup of $A$.  Let $|B|=\ell$.
Then we see that $n-\ell=2\lambda$ and  $3\lambda-\mu=n-2\ell$, i.e., $\lambda=\frac{n-\ell}{2}$ and $\mu=\frac{n+\ell}{2}$.
Also note that $\alpha\in B$ due to $\alpha\notin U$. As $f(\overline{X})=\overline{A}-\overline{B}-\overline{X}$,  by \eqref{equ::15} and \eqref{equ::16},
\begin{align}
&\overline{X}^2+\overline{X}(\overline{A}-\overline{B}-\overline{X})\alpha=\left(t-\frac{n+\ell}{2}\right)e-\ell\overline{X}+\frac{n-\ell}{2}\overline{A},\label{equ::19}\\
&\overline{X}^2+\overline{X}(\overline{A}-\overline{B}-\overline{X})=-\ell\overline{X}+\frac{n-\ell}{2}\overline{A}.\label{equ::20}
\end{align}
Since $|X|=\lambda=\frac{n-\ell}{2}$, $\alpha\in B$, and $B$ is a subgroup of $A$, from  \eqref{equ::19} and \eqref{equ::20} we obtain
\begin{align}
&\overline{X}^2-\overline{X}^2\alpha=\left(t-\frac{n+\ell}{2}\right)e,\label{equ::21}\\
&\overline{X}\cdot\overline{B}=\ell\overline{X}.\label{equ::22}
\end{align}
Clearly, the equality in \eqref{equ::22} implies that $X$ is a union of some cosets of $B$. Furthermore, by applying the trivial character to both sides of \eqref{equ::21}, we obtain $|X|^2-|X|^2=0=t-\frac{n+\ell}{2}$. Hence, $t=\frac{n+\ell}{2}$. Therefore, the parameters of $\Gamma$ are given by $(2n,n-\ell,\frac{n+\ell}{2},\frac{n-\ell}{2},\frac{n+\ell}{2})$.

Conversely, if $B=A\setminus (X\cup f(X))$ is a subgroup of $A$ and $X$ is a union of some cosets of $B$ in $A$, then one can easily verify that $X$ satisfies  \eqref{equ::15} and \eqref{equ::16} with $t=\mu=\frac{n+\ell}{2}$ and $\lambda=\frac{n-\ell}{2}$, where $\ell=|B|$. Therefore, $\Gamma$ is a DSRG by Lemma \ref{lem::6}.
\end{proof}

In what follows, we focus on considering the situation that  $Y=X\cup \{e\}$, which is a bit more complicated.

\begin{theorem}\label{thm::5}
Suppose that $X\cap f(X)=\emptyset$ and $Y=X\cup \{e\}$. Then the Cayey graph $\Gamma$ is a DSRG with parameters $(2n,k,\mu,\lambda,t)$ if and only if $k=2|X|+1$, $\alpha=e$, $f(X)=X^{-1}$, $|X|=\lambda$, $t=\lambda+1$, $n=(\lambda+1)(\lambda+\mu)/\mu$,  and one of the following situations occurs:
\begin{enumerate}[(i)]
    \item $\mu=\lambda+1$, and $X\cup X^{-1}=A\setminus\{e\}$; 
    \item $\mu\leq \lambda$, and $\mathrm{Cay}(G,X\cup X^{-1})$ is a SRG with parameters $(n,2\lambda,\lambda+\mu-2,2\mu)$ for which $X$ satisfies the equation 
\begin{equation}\label{equ::23}
    \overline{X}^2+(\lambda-\mu)\overline{X^{-1}}=\overline{X^{-1}}^2+(\lambda-\mu)\overline{X}.
\end{equation}
\end{enumerate}
In particular, the graph  $\mathrm{Cay}(G,X\cup X^{-1})$ in case (ii) has the least eigenvalue $-2$.
\end{theorem}
\begin{proof}
Suppose that $\Gamma$ is a DSRG with parameters $(2n,k,\mu,\lambda,t)$. Then $k=2|X|+1$. Putting $Y=X\cup \{e\}$ in  \eqref{equ::2} and \eqref{equ::3}, we obtain
\begin{align}
&\overline{X}(\overline{X}+f(\overline{X}))\alpha+(\overline{X}+f(\overline{X}))\alpha+\alpha=(t-\mu)e+(\lambda-\mu)\overline{X}+\mu\overline{A},\label{equ::24}\\
&\overline{X}(\overline{X}+f(\overline{X}))+\overline{X}+f(\overline{X})=(\lambda-\mu)e+(\lambda-\mu)\overline{X}+\mu\overline{A}.\label{equ::25}   
\end{align}
By applying the trivial character to both sides of \eqref{equ::24} and \eqref{equ::25}, we obtain $2|X|^2+2|X|+1=t-\mu+(\lambda-\mu)|X|+\mu n$ and $2|X|^2+2|X|=\lambda-\mu+(\lambda-\mu)|X|+\mu n$, respectively. Hence,  $t=\lambda+1$. Also, by applying $f$ to both sides of \eqref{equ::25}, we get
\begin{equation}\label{equ::26}  
f(\overline{X})(\overline{X}+f(\overline{X}))+\overline{X}+f(\overline{X})=(\lambda-\mu)e+(\lambda-\mu)f(\overline{X})+\mu\overline{A}.
\end{equation}
Combining \eqref{equ::25} and \eqref{equ::26} yields that 
\begin{align}
&(\overline{X}+f(\overline{X}))^2=2(\lambda-\mu)e+(\lambda-\mu-2)(\overline{X}+f(\overline{X}))+2\mu\overline{A},\label{equ::27}\\
&\overline{X}^2-f(\overline{X})^2=(\lambda-\mu)(\overline{X}-f(\overline{X})).\label{equ::28}   
\end{align}
Let $U=X\cup f(X)$. Then it follows from \eqref{equ::27} that 
\begin{equation}\label{equ::29}
\overline{U}^2=2\lambda e+(\lambda+\mu-2)\overline{U}+2\mu(\overline{A}-e-\overline{U}).
\end{equation}
Suppose $\Gamma'=\mathrm{Cay}(A,U)$.
By Theorem \ref{thm::1},  $\Gamma'$ is an undirected simple graph, and  hence $U=U^{-1}$. By comparing the coefficient of $e$ in both sides of \eqref{equ::29}, we obtain $|X|=\lambda$. Since $t=\lambda+1=|X|+1=|Y|$, we see that $X\cap X^{-1}=\emptyset$, and hence $f(X)=X^{-1}$. By applying the trivial character to both sides of \eqref{equ::29}, we obtain $n=(\lambda+1)(\lambda+\mu)/\mu$.  Combining Lemma \ref{lem::1} and  \eqref{equ::29}, we conclude that either $U=A\setminus\{e\}$, or $\Gamma'$ is a SRG with parameters $(n,2\lambda,\lambda+\mu-2,2\mu)$. 

If $U=A\setminus\{e\}$, then $\alpha=e$ because $\alpha\notin X\cup f(X)$ due to $X\cap f(X)=\emptyset$. Moreover, since $n=2\lambda+1=(\lambda+1)(\lambda+\mu)/\mu$, we have $\mu=\lambda+1$, and so $\Gamma$ has the parameters $(4\lambda+2,2\lambda+1,\lambda+1,\lambda,\lambda+1)$, where $\lambda=|X|$. This proves (i).

If $\Gamma'$ is a SRG with parameters $(n,2\lambda,\lambda+\mu-2,2\mu)$, then $\lambda\geq \mu$, and we assert that $\alpha=e$. In fact, if $\alpha\neq e$, from \eqref{equ::25} we see that $e$ occurs exactly $\lambda$ times in $\overline{X}(\overline{X}+f(\overline{X}))$, which implies that $\alpha$ occurs exactly $\lambda$ times in $\overline{X}(\overline{X}+f(\overline{X}))\alpha$. Since $e \notin X \cup f(X)$, $\alpha$ cannot occur in $(\overline{X}+f(\overline{X}))\alpha$. Recall that $t=\lambda+1$. By comparing the coefficient of $\alpha$ in both sides of \eqref{equ::24}, we obtain $\lambda+1=\mu$, contray to $\lambda\geq \mu$.  Thus (ii) holds according to \eqref{equ::28}.  Furthermore, by \cite[Theorem 9.1.3]{BH11}, it is easy to verify that $\Gamma'$ has the least eigenvalue $-2$.

Conversely, if (i) or (ii) holds, it is routine to check that the equalities in \eqref{equ::24} and \eqref{equ::25} hold. Therefore, $\Gamma$ is a DSRG  with parameters $(2n,k,\mu,\lambda,t)$ by Lemma \ref{lem::6}.
\end{proof}

In \cite{ADJ17}, Abdollahi, van Dam and Jazaeri proved that every Cayley SRG with least eigenvalue
$-2$ must be one of the following graphs: 
\begin{enumerate}[(i)]
\item the Clebsch graph; 
\item the Shrikhande graph; 
\item the Schl\"{a}fli graph;
\item the cocktail party graph $CP(n)$, with $n\geq 2$;
\item the triangular graph $T(n)$, with $n=4$, or $n\equiv 3\pmod 4$ and $n$ a prime power, $n>4$;
\item the lattice graph $L_2(n)$, with $n\geq 2$.
\end{enumerate}
Therefore, the graph $\mathrm{Cay}(A,X\cup X^{-1})$ in Theorem \ref{thm::5} (ii) should be one of these graphs listed in (i)--(vi). This might help us to determine the structure of $X$. However, it seems difficult to determine all possible abelian groups $A$ and all possible connection sets $S$ such that $\mathrm{Cay}(A,S)$ is isomorphic to one of the above graphs. For this reason, in the following, we only focus on the situation that $A$ is a cyclic group. 

Recall that a \textit{circulant graph} is a Cayley graph over a cyclic group. To achieve our goal, we need a classic result about the characterization of circulant SRGs. 

\begin{lemma}[\cite{BM79,M84}]\label{lem::7}
A circulant graph of order $n$ is a SRG if and only if it is isomorphic to one of the following graphs:
\begin{enumerate}[(i)]
\item the complete multipartite graph $K_{t\times m}$ with $tm=n$;
\item the Paley graph $P(n)$, where $n\equiv 1\pmod 4$ is prime.
\end{enumerate}
\end{lemma}

\begin{corollary}\label{cor::5}
Suppose that $A=\langle a\mid a^n=1\rangle$, $X\cap f(X)=\emptyset$ and $Y=X\cup \{e\}$. Then the Cayley graph $\Gamma$ is a DSRG if and only if $\alpha=e$, $f(X)=X^{-1}$, and $X=a^T:=\{a^i\mid i\in T\}$, where $T$ is a subset of $\mathbb{Z}_n\setminus\{0\}$ such that  $T\cap (-T)=\emptyset$ and either $T\cup (-T)=\mathbb{Z}_n\setminus\{0\}$, or $T\cup (-T)=\mathbb{Z}_n\setminus\{0,n/2\}$ and $T=n/2-T$ ($n$ is even).
\end{corollary}
\begin{proof}
Assume that $\Gamma$ is a DSRG. By Theorem \ref{thm::5}, we have $\alpha=e$, $f(X)=X^{-1}$, and either $X\cup X^{-1}=A\setminus \{e\}$, or $\Gamma':=\mathrm{Cay}(X\cup X^{-1})$ is a SRG with parameters $(n=(\lambda+1)(\lambda+\mu)/\mu,2\lambda,\lambda+\mu-2,2\mu)$ for which $X$ satisfies the equality in \eqref{equ::23},
where $\lambda=|X|$ and $\mu\leq \lambda$. Suppose $X=a^T$ with $T\subseteq \mathbb{Z}_n\setminus\{0\}$. Clearly, $T\cap (-T)=\emptyset$.   For the former case, we obtain  $T\cup (-T)=\mathbb{Z}_n\setminus\{0\}$, as required. For the later case, since $\Gamma'$ is a  circulant SRG of order $n$, by Lemma \ref{lem::7}, 
$\Gamma'$ must be the Paley graph $P(n)$, where $n\equiv 1\pmod 4$ is prime, or the complete multipartite graph $K_{t\times m}$ with $tm=n$. It is known that the Payley graph $P(n)$ has the parameters $(n,(n-1)/2,(n-5)/4,(n-1)/4)$ (cf. \cite[Proposition 9.1.1]{BH11}). If $\Gamma'=P(n)$, then we must have $n=9$, contrary to the fact that $n$ is prime. Therefore, $\Gamma'=K_{t\times m}$ with $tm=n$. Note that the complete multipartite graph $K_{t\times m}$ ($tm=n$) is a SRG with parameters $(n,n-m,n-2m,n-m)$. Then we have $m=2$, $\mu=\lambda$, $n=2\lambda+2$, and hence $X\cup X^{-1}=A\setminus \{e,a^{\frac{n}{2}}\}$, i.e., $T\cup (-T)=\mathbb{Z}_n\setminus\{0,n/2\}$ ($n$ is even). Furthermore, it easy to verify that  the equality in \eqref{equ::23} is actually equivalent to $T=n/2-T$. This proves the necessity. Conversely, it is easy to verify the sufficiency by using Theorem \ref{thm::5}. 
\end{proof}

\begin{remark}\label{rem::3}
\emph{If $G$ is a dihedral group, then Corollary \ref{cor::5} determines all Cayley DSRGs of the from $\mathrm{Cay}(G,X\cup X\beta\cup \{\beta\})$ with $X\cap X^{-1}=\emptyset$. However, if $G$ is a dicyclic group of order at least $8$, then $\alpha\neq e$ in $G$, and so Corollary \ref{cor::5} implies that there are no Cayley DSRGs of the form $\mathrm{Cay}(G,X\cup X\beta\cup \{\beta\})$ with $X\cap X^{-1}=\emptyset$.}
\end{remark}

At the end of this section, we construct a class of Cayley DSRGs over $G$ with $X\cap f(X)\neq \emptyset$.

\begin{theorem}\label{thm::6}
Let $n$ be a positive even integer. Suppose that $X$ is a subset of $A\setminus\{e\}$ satisfying the following conditions:
\begin{enumerate}[(i)]
\item $B=A\setminus(X\cup f(X))$ is a subgroup of $A$ such that $\alpha\in B$,
\item $X$ is a union of some cosets of $B$,
\item $X\cap f(X)=aB$ and $X\cup aX=A$ for some $a\in A$.
\end{enumerate}
Then $\Gamma=\mathrm{Cay}(G,X\cup X\beta)$ is a DSRG with parameters $(2n,n,\frac{n}{2}+\ell,\frac{n}{2}-\ell,\frac{n}{2}+\ell)$, where $\ell=|B|$. 
\end{theorem}
\begin{proof}
By assumption, we have  $\overline{X}+\overline{f(X)}=\overline{A}-\overline{B}+a\overline{B}$, and $|X|=\frac{n}{2}$. Let $\ell=|B|$. According to (i)-(iii), we see that
\begin{equation*}
\begin{aligned}
\overline{X}^2+\overline{X}f(\overline{X})\alpha&=\overline{X}^2+\overline{X}f(\overline{X})\\
&=\overline{X}^2+\overline{X}(\overline{A}-\overline{B}+a\overline{B}-\overline{X})\\
&=\overline{X}\cdot\overline{A}-\overline{X}\cdot\overline{B}+\overline{X}\cdot a\overline{B}\\
&=|X|\overline{A}-|B|\overline{X}+|B|\cdot a\overline{X}\\
&=|X|\overline{A}-|B|\overline{X}+|B|\cdot (\overline{A}-\overline{X})\\
&=-2|B|\overline{X}+(|X|+|B|)\overline{A}\\
&=-2\ell\overline{X}+\left(\frac{n}{2}+\ell\right)\overline{A}.
\end{aligned}
\end{equation*}
By Lemma \ref{lem::6}, the result follows.
\end{proof}

\begin{remark}\label{rem::4}
\emph{
It is worth mentioning that the construction in Theorem \ref{thm::6} extends the one given in \cite[Construction 6.6]{HZ19}.}
\end{remark}

\section{Further research}
Let $G$ be a non-abelian group with an abelian subgroup of index $2$. In this paper, we give some necessary conditions for a Cayley graph over $G$ to be directed strongly regular, and characterize the Cayley DSRGs over $G$ satisfying specified conditions. In particular, we determine Cayley DSRGs over $G$ with parameters $(2n,n-1,\frac{n-1}{2},\frac{n-3}{2},\frac{n-1}{2})$ (see Theorem \ref{thm::3}), and provide partial characterization for Cayley DSRGs over $G$ whose parameters are of the form:
\begin{enumerate}[(F1)]
    \item $(2n,n-\ell,\frac{n-\ell}{2},\frac{n-3\ell}{2},\frac{n-\ell}{2})$, or
    \item $(2n,n,\frac{n+\ell}{2},\frac{n-\ell}{2},\frac{n+\ell}{2})$, or
    \item $(2n,n,\frac{n}{2}+\ell,\frac{n}{2}-\ell,\frac{n}{2}+\ell)$.
\end{enumerate}
Motivated by these results, we pose the following problem for further research.
\begin{problem}\label{prob}
Determine the DSRGs (or Cayley DSRGs) whose parameters are of the form (F1), (F2) or (F3).
\end{problem}

\section*{Acknowledgments}
X. Huang is supported by National Natural Science Foundation of China (Grant No. 11901540). L. Lu is supported by National Natural Science Foundation of China (Grant No. 12001544) and  Natural Science Foundation of Hunan Province  (Grant No. 2021JJ40707).

\end{document}